\documentclass[10pt,leqno]{amsart}
\usepackage{amssymb,amsthm}
\usepackage{amsmath}
\usepackage{setspace}
\usepackage{dcpic,pictexwd}
\usepackage[all]{xy}
\usepackage[pdftex]{graphicx}
\usepackage{mathrsfs}
\usepackage[pdftex]{hyperref}
\usepackage{bbm}

\theoremstyle{definition}
 \newtheorem{definition}{Definition}[section]

\theoremstyle{plain}
 \newtheorem{proposition}[definition]{Proposition}

\theoremstyle{plain}
 \newtheorem{theorem}[definition]{Theorem}
 
\theoremstyle{definition}

\theoremstyle{plain}

\theoremstyle{plain}

\theoremstyle{remark}

\theoremstyle{definition}

\theoremstyle{plain}

\newcommand{\Ext}{\mathrm{Ext}}
\newcommand{\End}{\mathrm{End}}

\newcommand{\Hom}{\mathrm{Hom}}

\newcommand{\Ca}{\mathcal{C}}
\newcommand{\Fun}{\mathrm{F}}

\newcommand{\Def}{\mathrm{Def}}
\newcommand{\Sets}{\mathrm{Sets}}

\newcommand{\Ob}{\mathrm{Ob}}

\newcommand{\SEnd}{\underline{\End}}
\newcommand{\A}{\Lambda}

\newcommand{\m}{\mathfrak{m}}
\renewcommand{\k}{\Bbbk}

\newcommand{\invlim}{\varprojlim}

\hypersetup{
    bookmarks=true,         % show bookmarks bar?
    unicode=false,          % non-Latin characters in AcrobatÕs bookmarks
    pdftoolbar=true,        % show AcrobatÕs toolbar?
    pdfmenubar=true,        % show AcrobatÕs menu?
    pdffitwindow=false,     % window fit to page when opened
    pdfstartview={FitH},    % fits the width of the page to the window
    pdftitle={On deformations of Gorenstein-projective modules over the dual numbers},    % title
    pdfauthor={Jose A. Velez-Marulanda},     % author
    pdfsubject={},   % subject of the document
    pdfcreator={Jose A. Velez-Marulanda},   % creator of the document
    pdfproducer={Jose A. Velez-Marulanda}, % producer of the document
    pdfnewwindow=true,      % links in new window
    colorlinks=true,       % false: boxed links; true: colored links
    linkcolor=red,          % color of internal links
    citecolor=blue,        % color of links to bibliography
    filecolor=magenta,      % color of file links
    urlcolor=blue           % color of external links
}
\title[A short note on deformations of (strongly) Gorenstein-projective modules over the dual numbers]{A short note on deformations of (strongly) Gorenstein-projective modules over the dual numbers}

\author[V\'elez-Marulanda]{Jos\'e A. V\'elez-Marulanda}
\address{Department of Applied Mathematics \& Physics, Valdosta State University, Valdosta, GA,  United States of America}
\email{javelezmarulanda@valdosta.edu}
\address{Facultad de Matem\'aticas e Ingenier\'{\i}as, Fundaci\'on Universitaria Konrad Lorenz, Bogot\'a D.C.,  Colombia}
\email{josea.velezm@konradlorenz.edu.co}

\author{H\'ector Su\'arez}
\address{Escuela de Matem\'aticas y Estad\'istica, Universidad Pedag\'ogica y Tecnol\'ogica de Colombia - Sede Tunja}
\curraddr{University Campus}
\email{hector.suarez@uptc.edu.co}

\begin{document}
\renewcommand{\labelenumi}{\textup{(\roman{enumi})}}
\renewcommand{\labelenumii}{\textup{(\roman{enumi}.\alph{enumii})}}
\numberwithin{equation}{section}

\begin{abstract}
Let $\k$ be a field of arbitrary characteristic, and let $\A$ be a finite dimensional $\k$-algebra. In this short note we prove that if $V$ is a finitely generated strongly Gorenstein-projective left $\A$-module whose stable endomorphism ring $\SEnd_\A(V)$ is isomorphic to $\k$, then $V$ has an universal deformation ring $R(\A,V)$ isomorphic to the ring of dual numbers $\k[\epsilon]$ with $\epsilon^2=0$. As a consequence, we obtain the following result. Assume that $Q$ is a finite connected acyclic quiver, let $\k Q$ be the corresponding path algebra and let $\A = \k Q[\epsilon] = \k Q\otimes_\k \k[\epsilon]$. If $V$ is a finitely generated Gorenstein-projective left $\A$-module with $\SEnd_\A(V) = \k$, then $V$ has an universal deformation ring $R(\A,V)$ isomorphic to $\k[\epsilon]$. 

\end{abstract}
\maketitle

\section{Introduction}\label{int}

Throughout this note, we assume that $\k$ is a fixed field of arbitrary characteristic. We denote by $\widehat{\Ca}$ the category of all complete local commutative Noetherian $\k$-algebras with residue field $\k$. In particular, the morphisms in $\widehat{\Ca}$ are continuous $\k$-algebra homomorphisms that induce the identity map on $\k$.  Let $\A$ be a fixed finite dimensional $\k$-algebra, and let $R$ be a fixed but arbitrary object in $\widehat{\Ca}$.  We denote by $R\A$ the tensor product of $\k$-algebras $R\otimes_\k\A$. Note that if $R$ is an Artinian object in $\widehat{\Ca}$, then $R\A$ is also Artinian on both sides. In this note, we assume all our modules to be finitely generated. Recall that $\A$ is said to be a {\it Gorenstein} $\k$-algebra provided that $\A$ has finite injective dimension as a left and right $\A$-module (see \cite{auslander2}). In particular, algebras of finite global dimension as well as self-injective algebras are Gorenstein. 

Let $V$ be a fixed left $\A$-module. In \cite[Prop. 2.1]{blehervelez}, F. M. Bleher and the author proved that $V$ has a well-defined versal deformation ring $R(\A,V)$ in $\widehat{\Ca}$, which is universal provided that the endomorphism ring of $V$, namely $\End_\A(V)$, is isomorphic to $\k$. Moreover, they also proved that versal deformation rings are preserved under Morita equivalences (see \cite[Prop. 2.5]{blehervelez}).  Following \cite{enochs0,enochs}, we say that $V$ is {\it Gorenstein-projective} provided that there exists an acyclic complex of projective left $\A$-modules 
\begin{equation*}
P^\bullet: \cdots\to P^{-2}\xrightarrow{f^{-2}} P^{-1}\xrightarrow{f^{-1}} P^0\xrightarrow{f^0}P^1\xrightarrow{f^1}P^2\to\cdots
\end{equation*}  
such that $\Hom_\A(P^\bullet, \A)$ is also acyclic and $V=\mathrm{coker}\,f^0$. In particular, every projective left $\A$-module is Gorenstein-projective. Finitely generated Gorenstein-projective $\A$-modules are also known in the literature as modules of {\it Gorenstein dimension zero} or {\it totally reflexive} (see \cite{auslander4},\cite[\S 2]{avramov2}).  We denote by $\A$-Gproj the category of Gorenstein-projective left $\A$-modules, and by $\A\textup{-\underline{Gproj}}$ its stable category, i.e. $\A$-Gproj modulo the morphisms that factor through a projective left $\A$-module.  It is well-known that $\A$-Gproj is a Frobenius category in the sense of \cite[Chap. I, \S 2.1]{happel} and consequently,  $\A\textup{-\underline{Gproj}}$ is a triangulated category (in the sense of \cite{verdier}). Moreover, if $V$ is non-projective Gorenstein-projective, then for all $i\geq 0$, the $i$-th syzygy $\Omega_\A^iV$ is also non-projective Gorenstein-projective, and $\Omega$ induces an autoequivalence $\Omega: \A\textup{-\underline{Gproj}}\to \A\textup{-\underline{Gproj}}$ (see  \cite[Chap. I, \S2.2]{happel}), where $\Omega V$ is the kernel of a projective left $\A$-module cover $P(V)\to V$, which is unique up to isomorphism. Following \cite[Def. 2.1 \& Prop. 2.9]{bennis}, $V$ is said to be {\it strongly Gorenstein projective} if there exists a short exact sequence 
\begin{equation*}
0\to V\to P\to V\to 0, 
\end{equation*}
where $P$ is a projective left $\A$-module and $\Ext^1(V,P')=0$ for all projective left $\A$-modules $P'$.   Note in particular that every strongly Gorenstein-projective left $\A$-module is also Gorenstein-projective. 
These strongly Gorenstein-projective modules are being well-studied by many authors under different contexts (see e.g. \cite{ju} and its references). 

In \cite[Thm. 1.2]{bekkert-giraldo-velez}, it was proved that if  $V$ is a Gorenstein-projective left $\A$-module that has stable endomorphism ring $\SEnd_\A(V)$ isomorphic to $\k$, then the versal deformation ring $R(\A,V)$ is universal. Moreover, in  \cite[Thm. 1.2]{velez4}  in was proved that under this situation, we also have $\Omega V$ has also a universal deformation ring $R(\A, \Omega V)$ that is also isomorphic to $R(\A, V)$, and that  universal deformation rings of  Gorenstein-projective modules whose stable endomorphism ring is isomorphic to $\k$ are preserved under {\it singular equivalences of Morita type with level} between Gorenstein $\k$-algebras. These singular equivalences of Morita type with level were introduced by Z. Wang  in \cite[Def. 2.1]{wang} in order to generalize the concept of {\it singular equivalences of Morita type} introduced by Chen and L. G. Sun in \cite{chensun} (and which were further studied by G. Zhou and A. Zimmermann in \cite{zhouzimm}) as a way of generalizing the concept of stable equivalences of Morita type (as introduced by M. Brou\'e in \cite{broue}).  We refer the reader to \cite{qin} to obtain more information and examples regarding these singular equivalences of Morita type with level.

Let $\A= \k Q[\epsilon] = \k Q\otimes_\k \k[\epsilon]$ where $\k Q$ is the path algebra of a finite, connected and acyclic quiver $Q$ (see e.g. \cite[Chap. III, \S 1]{auslander}), and $\k[\epsilon]$ is the ring of dual numbers with $\epsilon^2=0$. In \cite{ringel-zhang3}, C. M. Ringel and P. Zhang studied Gorenstein-projective modules over $\A$. In particular, it follows from the discussion in \cite[\S 1]{ringel-zhang3}, that $\A$-Gproj coincides with the category of perfect differential $\k Q$-modules, which they denote by $\mathrm{diff}_\textup{perf}(\k Q)$. As a consequence, $\mathrm{diff}_\textup{perf}(\k Q)$ is a Frobenius category whose stable category $\underline{\mathrm{diff}}_\textup{perf}(\k Q)$ is a triangulated category that is equivalent to the orbit category $\mathcal{D}^b(\k Q\textup{-mod})/[1]$, where  $\mathcal{D}^b(\k Q\textup{-mod})$ is the bounded derived category of $\k Q$ and $[1]$ is the suspension functor on $\mathcal{D}^b(\k Q\textup{-mod})$. Since $\A$ is a Gorenstein $\k$-algebra by \cite[\S 2.2]{ringel-zhang3} (see also Proposition \ref{prop1}), it follows that the {\it singularity category} of $\k Q$ (see e.g. \cite[Thm. 6.2.5]{krause4}) is equivalent to $\mathcal{D}^b(\k Q\textup{-mod})/[1]$ as triangulated categories. It as also proved in \cite[\S 4.12]{ringel-zhang3} (see also  Proposition \ref{prop2}) that if $V$ is a Gorenstein-projective left $\A$-module, then $V$ is also strongly Gorenstein-projective.

In this note we are interested in proving the following result.

\begin{theorem}\label{thm1}
Let $\A = \k Q[\epsilon]= \k Q \otimes_\k \k[\epsilon]$ where $\k Q$ is the path algebra of a finite, connected and acyclic quiver $Q$, and $\k[\epsilon]$ is the ring of dual numbers with $\epsilon^2=0$.  If $V$ is a Gorenstein-projective left $\A$-module with $\SEnd_\A(V)= \k$, then the universal deformation ring $R(\A,V)$ is isomorphic to $\k[\![t]\!]/(t^2)=\k[\epsilon]$.
\end{theorem}

By the above discussion, Theorem \ref{thm1} provides a classification of universal deformation rings of Gorenstein-projective modules up to singular equivalence of Morita type with level between Gorenstein $\k$-algebras over the dual numbers.  

In order to prove Theorem \ref{thm1}, we prove first the following result (see \S \ref{sec2}). 

\begin{theorem}\label{thm2}
Let $\A$ be a finite dimensional $\k$-algebra. If $V$ is a strongly Gorenstein-projective left $\A$-module such that $\SEnd_\A(V)=\k$, then the universal deformation ring $R(\A,V)$ is isomorphic to $\k[\![t]\!]/(t^2)$.
\end{theorem}

For basic concepts concerning Gorenstein-projective modules and representation theory of algebras, we refer the reader to  \cite{auslander,auslander3,holmH} (and their references). For an historical background and previous work on lifts and deformations of modules (in the sense of this article) under the context of group algebras, and self-injective $\k$-algebras, we invite the reader to look at the articles \cite{blehervelez,blehervelez2,bleher15,bekkert-giraldo-velez,velez2,velez3} and their references.   

\section{Preliminaries}\label{sec2}
Throughout this section we keep the notation introduced in \S \ref{int}. In particular, we assume that $\A$ is a finite dimensional $\k$-algebra.

\subsection{Lifts, deformations, and (uni)versal deformation rings}\label{sec21}

Let $V$ be a left $\A$-module and let $R$ be a fixed but arbitrary object in $\widehat{\Ca}$. A {\it lift} $(M,\phi)$ 
of $V$ over $R$ is a finitely generated left $R\A$-module $M$ 
that is free over $R$ 
together with an isomorphism of $\A$-modules $\phi:\k\otimes_RM\to V$. Two lifts $(M,\phi)$ and $(M',\phi')$ over $R$ are {\it isomorphic} 
if there exists an $R\A$-module 
isomorphism $f:M\to M'$ such that $\phi'\circ (\mathrm{id}_\k\otimes_R f)=\phi$.
If $(M,\phi)$ is a lift of $V$ over $R$, we  denote by $[M,\phi]$ its isomorphism class and say that $[M,\phi]$ is a {\it deformation} of $V$ 
over $R$. We denote by $\Def_\A(V,R)$ the 
set of all deformations of $V$ over $R$. The {\it deformation functor} corresponding to $V$ is the 
covariant functor $\widehat{\Fun}_V:\widehat{\Ca}\to \Sets$ defined as follows: for all objects $R$ in $\widehat{\Ca}$, define $\widehat{\Fun}_V(R)=\Def_
\A(V,R)$, and for all morphisms $\theta:R\to 
R'$ in $\widehat{\Ca}$, 
let $\widehat{\Fun}_V(\theta):\Def_\A(V,R)\to \Def_\A(V,R')$ be defined as $\widehat{\Fun}_V(\theta)([M,\phi])=[R'\otimes_{R,\theta}M,\phi_\theta]$, 
where $\phi_\theta: \k\otimes_{R'}
(R'\otimes_{R,\theta}M)\to V$ is the composition of $\A$-module isomorphisms 
\[\k\otimes_{R'}(R'\otimes_{R,\theta}M)\cong \k\otimes_RM\xrightarrow{\phi} V.\]

Suppose there exists an object $R(\A,V)$ in $\widehat{\Ca}$ and a deformation $[U(\A,V), \phi_{U(\A,V)}]$ of $V$ over $R(\A,V)$ with the 
following property. For all objects $R$ in $\widehat{\Ca}$ and for all deformations $[M,\phi]$ of $V$ over $R$, there exists a morphism $\psi_{R(\A,V),R,[M,\phi]}:R(\A,V)\to R$ 
in $\widehat{\Ca}$ such that 
\[\widehat{\Fun}_V(\psi_{R(\A,V),R,[M,\phi]})[U(\A,V), \phi_{U(\A,V)}]=[M,\phi],\]
and moreover, $\psi_{R(\A,V),R,[M,\phi]}$ is unique if $R$ is the ring of dual numbers $\k[\epsilon]$ with $\epsilon^2=0$.  Then $R(\A,V)$ and $
[U(\A,V),\phi_{U(\A,V)}]$ are called the {\it versal deformation ring} and {\it versal deformation} of $V$, respectively. If the morphism $
\psi_{R(\A,V),R,[M,\phi]}$ is unique for all $R\in\Ob(\widehat{\Ca})$ and deformations $[M,\phi]$ of $V$ over $R$, then $R(\A,V)$ and $[U(\A,V),\phi_{U(\A,V)}]$ are 
called the {\it universal deformation ring} and the {\it universal deformation} of $V$, respectively.  In other words, the universal deformation 
ring $R(\A,V)$ represents the deformation functor $\widehat{\Fun}_V$ in the sense that $\widehat{\Fun}_V$ is naturally isomorphic to the $\Hom$ 
functor $\Hom_{\widehat{\Ca}}(R(\A,V),-)$. 

We denote by $\Fun_V$  the restriction of $\widehat{\Fun}_V$ to the full subcategory of Artinian objects in $\widehat{\Ca}$. Following \cite[\S 2.6]{sch}, we call the set $\Fun_V(\k[\epsilon])$ the tangent space of $\Fun_V$, which has a structure of a $\k$-vector space by \cite[Lemma 2.10]{sch}. 
It was proved in \cite[Prop. 2.1]{blehervelez} that $\Fun_V$ satisfies the Schlessinger's criteria \cite[Thm. 2.11]{sch}, that there exists an isomorphism of $\k$-vector spaces 
\begin{equation}\label{hoch}
\Fun_V(\k[\epsilon])\to \Ext_\A^1(V,V),
\end{equation}
and that $\widehat{\Fun}_V$ is continuous in the sense of \cite[\S 14]{mazur}, i.e. for all objects $R$ in $\widehat{\Ca}$, we have  
\begin{equation}\label{cont}
\widehat{\Fun}_V(R)=\invlim_n \Fun_V(R/\m_R^n), 
\end{equation}
where $\m_R$ denotes the unique maximal ideal of $R$. Consequently, $V$ has always a well-defined versal deformation ring $R(\A,V)$ which is also universal provided that $\End_\A(V)$ is isomorphic to $\k$. It was also proved in \cite[Prop. 2.5]{blehervelez} that versal deformation rings are invariant under Morita equivalences between finite dimensional $\k$-algebras. It follows from the isomorphism of $\k$-vector spaces (\ref{hoch}) that if $\dim_\k \Ext_\A^1(V,V)=r$, then the versal deformation ring $R(\A,V)$ is isomorphic to a quotient algebra of the power series ring  $\k[\![t_1,\ldots,t_r]\!]$ and $r$ is minimal with respect to this property. In particular, if $V$ is a left $\A$-module such that $\Ext_\A^1(V,V)=0$, then $R(\A,V)$ is universal and isomorphic to $\k$ (see \cite[Remark 2.1]{bleher15} for more details). As already noted in \S\ref{int}, it follows by \cite[Thm. 1.2 (ii)]{bekkert-giraldo-velez} that if $V$ is a Gorenstein-projective left $\A$-module with stable endomorphism $\SEnd_\A(V)$ isomorphic to $\k$, then the versal deformation ring is $R(\A,V)$ is universal. Moreover, if $P$ is a projective left $\A$-module, then the versal deformation ring $R(\A,V\oplus P)$ is also universal and isomorphic to $R(\A,V)$. 

\subsection{(Strongly) Gorenstein projective modules over the dual numbers}

Assume next that $\A = \k Q[\epsilon]$ where $\k Q$ is the path algebra of a finite, connected and acyclic quiver $Q$, and $\k[\epsilon]$ is the ring of dual numbers with $\epsilon^2=0$, and let $V$ a left $\A$-module. We say that $V$ is {\it torsionless} provided that $V$ is a submodule of a projective left $\A$-module. The following result follows from \cite[\S 2.2]{ringel-zhang3}.

\begin{proposition}\label{prop1}
The $\k$ algebra $\A=\k Q[\epsilon]$ is Gorenstein and the following conditions are equivalent for all left $\A$-modules $V$:
\begin{enumerate}
\item $V$ is Gorenstein-projective;
\item $V$ is torsionless;
\item $\Ext_\A^1(V,\A)=0$.
\end{enumerate}
\end{proposition}

The following result follows from \cite[\S 4.12]{ringel-zhang3}.

\begin{proposition}\label{prop2}
If $V$ is a torsionless left $\k Q[\epsilon]$-module, then $V$ is strongly Gorenstein-projective. 
\end{proposition}

\section{Proof of main results}\label{sec2}

In the following, we prove Theorem \ref{thm2} by using an argument similar to the one used in the proof of \cite[Thm. 5.2]{bekkert-giraldo-velez}.
\begin{proof}[Proof of Theorem \ref{thm2}]
Assume that $V$ is strongly Gorenstein-projective left $\A$-module. Then there exists a short exact sequence of left $\A$-modules 
\begin{equation*}
0\to V \xrightarrow{\iota} P\xrightarrow{\pi} V \to 0, 
\end{equation*}
where $P$ is a projective left $\A$-module. It follows that $P$ defines a non-trivial lift of $V$ over the ring of dual numbers $\k[\![t]\!]/(t^2)$, where the action of $t$ is given by $\iota \circ \pi$. This implies that there exists a unique surjective $\k$-algebra morphism $\psi: R(\A, V)\to \k[\![t]\!]/(t^2)$ in $\widehat{\Ca}$ corresponding to the deformation defined by $P$. We claim that $\psi$ is an isomorphism, for if we suppose otherwise, then there exists a surjective $\k$-algebra homomorphism $\psi_0:R(\A,V)\to \k[\![t]\!]/(t^3)$ in $\hat{\Ca}$ such that $\pi'\circ \psi_0=\psi$, where $\pi':\k[\![t]\!]/(t^3)\to \k[\![t]\!]/(t^2)$ is the natural projection.  Let $M_0$ be a $\k[\![t]\!]/(t^3)\A$-module which defines a lift of $V$ over $\k[\![t]\!]/(t^3)$ corresponding to $\psi_0$.  Let $(U(\A,V),\phi_{U(\A,V)})$ be a lift of $V$ over $R(\A,V)$ that defines the universal deformation of $V$. Then $M_0\cong \k[\![t]\!]/(t^3)\otimes_{R(\A,V), \psi_0}U(\A,V)$. Note that $M_0/tM_0\cong V$ as $\A$-modules. On the other hand, we also have that 
\begin{align*}
P&\cong \k[\![t]\!]/(t^2)\otimes_{R(\A,V), \psi}U(\A,V)\\
&\cong \k[\![t]\!]/(t^2)\otimes_{\k[\![t]\!]/(t^3), \pi'}(\k[\![t]\!]/(t^3)\otimes_{R(\A,V), \psi_0}U(\A,V))\\
&\cong \k[\![t]\!]/(t^2)\otimes_{\k[\![t]\!]/(t^3),\pi'}M_0.
\end{align*}
Note that since $\ker \pi'=(t^2)/(t^3)$, we have $\k[\![t]\!]/(t^2)\otimes_{\k[\![t]\!]/(t^3),\pi'}M_0\cong M_0/t^2M_0$. Thus $P\cong M_0/t^2M_0$ as $\k[\![t]\!]/(t^2)\A$-modules. Consider the surjective $\k[\![t]\!]/(t^3)\A$-module homomorphism $g:M_0\to t^2M_0$ defined by $g(x)=t^2x$ for all $x\in M_0$. Since $M_0$ is free over $\k[\![t]\!]/(t^3)$, it follows that $\ker g =tM_0$ and thus $M_0/tM_0\cong t^2M_0$, which implies that $V\cong t^2M_0$. Hence we get a short exact sequence of $\k[\![t]\!]/(t^3)\A$-modules 
\begin{equation}\label{ext4}
0\to V\to M_0\to P\to 0.
\end{equation} 
Since $P$ is a projective left $\A$-module, it follows that (\ref{ext4}) splits as a short exact sequence of $\A$-modules. Hence $M_0=V\oplus P$ as $\A$-modules. Writing elements of $M_0$ as $(u,v)$  where $u\in V$ and $v\in P$, the $t$-action on $M_0$ is given as $t(u,v)=(\mu(v), tv)$ for some surjective $\A$-module homomorphism $\mu: P\to V$. Using $t^2v=0$ for all $v\in P$, we obtain that $t^2(u,v)=(\mu(tv),0)$ for all $u\in V$ and $v\in P$. Since $t^2M_0\cong V$, this means that the restriction of $\mu$ to $tP$ has to define an isomorphism $\tilde{\mu}:tP\to V$. Therefore, $\tilde{\mu}^{-1}$ provides a $\A$-module homomorphism splitting of $\mu$, which shows that $V$ is isomorphic to a direct summand of $P$. But then $V$ is itself projective, which contradicts our assumption that $\SEnd_\A(V)=\k$. Thus $\psi$ is a $\k$-algebra isomorphism and $R(\A,V)\cong \k[\![ t]\!]/(t^2)$. 
\end{proof}

It follows now that Theorem \ref{thm1} is a direct consequence of Propositions \ref{prop1}, \ref{prop2} and Theorem \ref{thm2}.

\section{Data availability statement}  
This manuscript has no associated data.

\bibliographystyle{amsplain}
\bibliography{DeformationsDualNumbers}

\end{document}